\documentclass{amsart}

\usepackage{lmodern} 
\usepackage{microtype} 
\usepackage[english]{babel} 
\usepackage{mathtools} 
\usepackage{mathrsfs} 
\usepackage[hidelinks]{hyperref} 
\usepackage[msc-links]{amsrefs}

\theoremstyle{plain} 
\newtheorem{theorem}{Theorem} 
\newtheorem{lemma}[theorem]{Lemma}

\begin{document} 
\title{Two footnotes to the F. \& M. Riesz theorem} 
\date{\today} 

\author{Ole Fredrik Brevig} 
\address{Department of Mathematics, University of Oslo, 0851 Oslo, Norway} 
\email{obrevig@math.uio.no}
\begin{abstract}
	We present a new proof of the F.~\&~M.~Riesz theorem on analytic measures of the unit circle $\mathbb{T}$ that is based the following elementary inequality: If $f$ is analytic in the unit disc $\mathbb{D}$ and $0 \leq r \leq \varrho < 1$, then
	\[\|f_r-f_\varrho\|_1 \leq 2 \sqrt{\|f_\varrho\|_1^2-\|f_r\|_1^2},\]
	where $f_r(e^{i\theta})=f(r e^{i\theta})$ and where $\|\cdot\|_1$ denotes the norm of $L^1(\mathbb{T})$. The proof extends to the infinite-dimensional torus $\mathbb{T}^\infty$, where it clarifies the relationship between Hilbert's criterion for $H^1(\mathbb{T}^\infty)$ and the F.~\&~M.~Riesz theorem. 
\end{abstract}


\maketitle

\section{Introduction} A finite complex Borel measure $\mu$ on the unit circle $\mathbb{T}$ is uniquely determined by the Fourier coefficients
\[\widehat{\mu}(k) = \int_0^{2\pi} e^{-ik\theta}\,d\mu(e^{i\theta}),\]
for $k$ in $\mathbb{Z}$. This assertion is a consequence of the fact that trigonometric polynomials are dense in $C(\mathbb{T})$ and duality in form of the Riesz representation theorem. The protagonist of the present note is the following well-known result due to F.~\&~M.~Riesz (see e.g.~\cite{Riesz1988}*{pp.~195--212}) on analytic measures of the unit circle.
\begin{theorem}\label{thm:fmriesz} 
	If $\mu$ is a finite complex Borel measure on $\mathbb{T}$ that satisfies $\widehat{\mu}(k) = 0$ for $k<0$, then $\mu$ is absolutely continuous. 
\end{theorem}

There are several proofs of Theorem~\ref{thm:fmriesz} of rather distinct flavor. The original proof of F.~\&~M.~Riesz relies on approximation (as does the short proof of {\O}ksendal \cite{Oksendal1971}), while the modern proofs use either Hilbert space techniques or the Poisson kernel. Should the reader desire a side-by-side comparison, we refer to the monograph of Koosis~\cite{Koosis1998} that contains all three variants. 

Our first footnote concerns a simplification to the proof based on the Poisson kernel, so let us recall the setup. The assumptions of Theorem~\ref{thm:fmriesz} ensure that the Poisson extension
\[\mathfrak{P}\mu(z) = \int_0^{2\pi} \frac{1-|z|^2}{|e^{i\theta}-z|^2}\,d\mu(e^{i\theta})\]
is analytic (whence $\mu$ is an ``analytic'' measure) in the unit disc $\mathbb{D}$, since it can be represented by an absolutely convergent power series at the origin. We also get from Fubini's theorem that
\[\int_0^{2\pi} \left|\mathfrak{P}\mu(r e^{i\theta})\right|\,\frac{d\theta}{2\pi} \leq \|\mu\|\]
for every $0\leq r < 1$, where $\|\mu\|$ denotes the total variation of $\mu$. 

In combination, these two assertions show that the function $f = \mathfrak{P}\mu$ is in the Hardy space $H^1(\mathbb{D})$. Let us define $f_r(e^{i\theta})=f(r e^{i\theta})$ for $0 \leq r < 1$. The last step in the proof of Theorem~\ref{thm:fmriesz} is to show that there is a function $f^\ast$ in $L^1(\mathbb{T})$ such that $\|f^\ast-f_r\|_1 \to 0$ as $r\to 1^-$. It would follow from this that $f^\ast = \mu$, since they have the same Fourier coefficients. This is where our proof diverges from the standard proofs, that first use Fatou's theorem to define $f^\ast$ as the boundary value function of $f$ and then establish that $f_r$ converges in norm to $f^\ast$. We will instead use the following result, which in particular means that Fatou's theorem is not required.
\begin{lemma}\label{lem:main} 
	If $f$ is analytic in $\mathbb{D}$ and $0 \leq r \leq \varrho < 1$, then
	\[\int_0^{2\pi} \big|f(re^{i\theta})-f(\varrho e^{i\theta})\big|\,\frac{d\theta}{2\pi} \leq 2 \sqrt{\left(\int_0^{2\pi} \big|f(\varrho e^{i\theta})\big|\,\frac{d\theta}{2\pi}\right)^2 - \left(\int_0^{2\pi} \big|f(r e^{i\theta})\big|\,\frac{d\theta}{2\pi}\right)^2}.\]
\end{lemma}

Theorem~\ref{thm:fmriesz} now follows at once. Lemma~\ref{lem:main} shows that if $f$ is in $H^1(\mathbb{D})$, then any sequence of functions $f_r$ with $r \to 1^-$ forms a Cauchy sequence in $L^1(\mathbb{T})$. From this point of view, Lemma~\ref{lem:main} should be considered a quantitative version of the qualitative assertion that $\|f^\ast-f_r\|_1 \to 0$ as $r \to 1^-$.

The proof of Lemma~\ref{lem:main} is elementary: it uses only finite Blaschke products, the triangle inequality, the Cauchy--Schwarz inequality, and orthogonality. It inspired by a result of Kulikov \cite{Kulikov2021}*{Lemma~2.1} that essentially corresponds to the case $r=0$. 

It would be interesting to know what the best constant $C$ in the estimate appearing Lemma~\ref{lem:main} is. Our result is that $C \leq 2$. Choosing $f(z)=1+\varepsilon z$ and $r=0$, then letting $\varepsilon \to 0^+$ shows that $C \geq \sqrt{2}$. It can be extracted from the proof of the main result in \cite{BS2024} that $C=\sqrt{2}$ is the best constant for $r=0$. A related problem of interest is to establish versions of Lemma~\ref{lem:main} where $L^p(\mathbb{T})$ takes the place of $L^1(\mathbb{T})$. 

Lemma~\ref{lem:main} also contains the fact that the radial means $r \mapsto \|f_r\|_1$ are increasing. From an historical point of view, let us recall that this answers the question posed by Bohr and Landau to Hardy~\cite{Hardy1915}, which led to the paper that is considered to mark the starting point of the theory. Lemma~\ref{lem:main} provides a simpler proof of this fact, which is typically established using convexity. However, the standard proofs yield the stronger assertion that $\log{r} \mapsto \log{\|f_r\|_1}$ is convex for $0<r<1$. 

Our second footnote concerns the (countably) infinite-dimensional torus
\[\mathbb{T}^\infty = \mathbb{T} \times \mathbb{T} \times \mathbb{T} \times \cdots,\]
that forms a compact abelian group under multiplication. Its dual group is $\mathbb{Z}^{(\infty)}$, the collection of compactly supported integer-valued sequences, and its normalized Haar measure $m_\infty$ coincides with the infinite product measure generated by the normalized Lebesgue arc length measure on $\mathbb{T}$.

The spaces $L^p(\mathbb{T}^\infty)$ contain a natural chain of subspaces that can be identified with $L^p(\mathbb{T}^d)$ for $d=1,2,3,\ldots$ and \emph{die Abschnitte} $\mathfrak{A}_d$ define bounded linear operators on $L^p(\mathbb{T}^\infty)$ that satisfy $\|\mathfrak{A}_1 f\|_p \leq \|\mathfrak{A}_2 f\|_p \leq \|\mathfrak{A}_3 f\|_p \leq \cdots \leq \|f\|_p$ for $f$ in $L^p(\mathbb{T}^\infty)$.

It follows from this that if $f$ is a function in $L^p(\mathbb{T}^\infty)$ and $f_d = \mathfrak{A}_d f$, then $(f_d)_{d\geq1}$ is a bounded sequence in $L^p(\mathbb{T}^\infty)$ that enjoys the \emph{chain property}
\[\mathfrak{A}_d f_{d+1} = f_d\]
for $d=1,2,3,\ldots$. The following fundamental questions arise naturally. 
\begin{enumerate}
	\item[(i)] If $f$ is a function in $L^p(\mathbb{T}^\infty$), then how does $\mathfrak{A}_d f$ tend to $f$ as $d \to \infty$? 
	\item[(ii)] Given a bounded sequence $(f_d)_{d\geq1}$ in $L^p(\mathbb{T}^\infty)$ that enjoys the chain property, is there a function $f$ in $L^p(\mathbb{T}^\infty)$ such that $f_d = \mathfrak{A}_d f$ for $d=1,2,3,\ldots$? 
\end{enumerate}

It is not difficult to prove that if $1 \leq p < \infty$, then answer to (i) is that the sequence $(\mathfrak{A}_d f)_{d\geq1}$ converges to $f$ in norm (see Theorem~\ref{thm:hilbertcritLp} below). If $1<p<\infty$, then a standard argument involving duality and the Banach--Alaoglu theorem shows that the answer to (ii) is affirmative. The conclusion is that in the strictly convex regime there is a one-to-one correspondence between functions in $L^p(\mathbb{T}^\infty)$ and bounded sequences in $L^p(\mathbb{T}^\infty)$ that enjoy the chain property. We refer to this type of result as \emph{Hilbert's criterion}, as the basic idea goes back to Hilbert~\cite{Hilbert1909}. 

It is well-known that Hilbert's criterion does not hold for $L^1(\mathbb{T}^\infty)$, although we have not found this explicitly stated in the literature. Let $z=(z_1,z_2,z_3,\ldots)$ be a point in the infinite polydisc $\mathbb{D}^\infty$ and consider the sequence $(f_d)_{d\geq1}$, where 
\begin{equation}\label{eq:poissonprod} 
	f_d(\chi) = \prod_{j=1}^d \frac{1-|z_j|^2}{|\chi_j-z_j|^2} 
\end{equation}
for $\chi$ on $\mathbb{T}^\infty$. It is not difficult to see that $(f_d)_{d\geq1}$ is bounded sequence in $L^1(\mathbb{T}^\infty)$ enjoying the chain property. However, a result of Cole and Gamelin~\cite{CG1986}*{Theorem~3.1} is equivalent to the assertion that there is a function $f$ in $L^1(\mathbb{T}^\infty)$ such that $f_d = \mathfrak{A}_d f$ for $d=1,2,3,\ldots$ if and only if $z$ is in $\mathbb{D}^\infty \cap \ell^2$. Choosing therefore a point $z$ in $\mathbb{D}^\infty \setminus \ell^2$, we see that (ii) has a negative answer for $p=1$.

Set $\mathbb{N}_0=\{0,1,2,\ldots\}$ and define the Hardy space $H^p(\mathbb{T}^\infty)$ as the closed subspace of $L^p(\mathbb{T}^\infty)$ consisting of the functions $f$ whose Fourier coefficients
\[\widehat{f}(\kappa) = \int_{\mathbb{T}^\infty} f(\chi) \,\overline{\chi^{\kappa}}\,dm_\infty(\chi)\]
are supported on $\mathbb{N}_0^{(\infty)}$. It turns out that Hilbert's criterion holds for $H^1(\mathbb{T}^\infty)$.
\begin{theorem}\label{thm:hilbertcrit} 
	\mbox{} 
	\begin{enumerate}
		\item[(i)] If $f$ is in $H^1(\mathbb{T}^\infty)$, then $\|f-\mathfrak{A}_d f\|_1 \to 0$ as $d \to \infty$. 
		\item[(ii)] If $(f_d)_{d\geq1}$ is a bounded sequence in $H^1(\mathbb{T}^\infty)$ that enjoys the chain property, then there is a function $f$ in $H^1(\mathbb{T}^\infty)$ such that $f_d = \mathfrak{A}_d f$ for $d=1,2,3,\ldots$. 
	\end{enumerate}
\end{theorem}

Note that a more general version of Theorem~\ref{thm:hilbertcrit}~(ii) can be extracted from work of Bourgain~\cite{Bourgain1983}*{Section~5}. In our context, Theorem~\ref{thm:hilbertcrit} was first enunciated by Aleman, Olsen, and Saksman \cite{AOS2019}*{Corollary~3}. To explain their approach, note that if $z$ is in $\mathbb{D}^\infty \setminus \ell^2$, then it follows from the result of Cole and Gamelin that the sequence $(f_d)_{d\geq1}$ with $f_d$ as in \eqref{eq:poissonprod} will converge weak-$\ast$ to a finite Borel measure $\mu$ on $\mathbb{T}^\infty$ that is not absolutely continuous (with respect to $m_\infty$). This leads us back to the F.~\&~M.~Riesz theorem on analytic measures, which in this context can be formulated as follows.
\begin{theorem}\label{thm:fmrieszinfty} 
	If $\mu$ is a finite complex Borel measure on $\mathbb{T}^\infty$ whose Fourier coefficients
	\[\widehat{\mu}(\kappa) = \int_{\mathbb{T}^\infty} \chi^{-\kappa}\,d\mu(\chi)\]
	are supported on $\mathbb{N}_0^{(\infty)}$, then $\mu$ is absolutely continuous. 
\end{theorem}

In view of the discussion above, it is plain that Theorem~\ref{thm:fmrieszinfty} implies Theorem~\ref{thm:hilbertcrit}~(ii). A stronger version of Theorem~\ref{thm:fmrieszinfty} goes back to Helson and Lowdenslager \cite{HL1958}. The current version is as stated by Aleman, Olsen, and Saksman \cite{AOS2019}*{Corollary~1}, who proved Theorem~\ref{thm:fmrieszinfty} after first establishing a version of Fatou's theorem in the infinite polydisc. The basic obstacle in this context is that the Poisson extension of $\mu$ is in general only defined on $\mathbb{D}^\infty \cap \ell^1$, and the main effort in \cite{AOS2019} is directed at obtaining a version of Fatou's theorem where $\mathbb{T}^\infty$ is approached from $\mathbb{D}^\infty \cap \ell^1$.

Our proof of the F.~\&~M.~Riesz theorem on $\mathbb{T}$ also leads to simpler proofs of Theorem~\ref{thm:hilbertcrit} and Theorem~\ref{thm:fmrieszinfty}, since we can avoid Fatou's theorem once we have established suitable extensions of Lemma~\ref{lem:main}. 

This line of reasoning also reveals that Theorem~\ref{thm:hilbertcrit} only uses the case $r=0$ of Lemma~\ref{lem:main}, while Theorem~\ref{thm:fmrieszinfty} requires the full result. Inspired by this and by the philosophy behind Hilbert's criterion, we find it natural to incorporate Theorem~\ref{thm:hilbertcrit} in the proof of Theorem~\ref{thm:fmrieszinfty}. Amusingly, this is the reverse direction to how the two results were established in \cite{AOS2019}.

\subsection*{Organization} The present note is comprised of three sections. Section~\ref{sec:main} is devoted to the proof of Lemma~\ref{lem:main}, while Section~\ref{sec:hilbertcrit} contains some expositional material and the proofs of Theorem~\ref{thm:hilbertcrit} and Theorem~\ref{thm:fmrieszinfty}.

\section{Proof of Lemma~\ref{lem:main}} \label{sec:main} 
By continuity, it is sufficient to consider only those $0 < \varrho <1$ such that $f$ does not vanish on the circle $|z|=\varrho$. Since $f$ is analytic in $\mathbb{D}$ it has only a finite number of zeros in $\varrho \mathbb{D}$. Let $(\alpha_n)_{n=1}^m$ denote these zeros (counting multiplicities) and form the finite Blaschke product
\[B(z) = \prod_{n=1}^m \frac{\varrho(\alpha_n-z)}{\varrho^2-\overline{\alpha_n}z}.\]
Note that $|B(z)|=1$ if $|z|=\varrho$. The function $F = f/B$ is analytic and non-vanishing when $|z| < \varrho+\varepsilon$ for some $\varepsilon>0$, due to the assumption that $f$ does not vanish on the circle $|z|=\varrho$. This means in particular that the functions $g = B F^{1/2}$ and $h = F^{1/2}$ are analytic for $|z|<\varrho+\varepsilon$ and that $f=gh$. We write
\[f_r(e^{i\theta})=f(r e^{i\theta}),\qquad g_r(e^{i\theta})=g(r e^{i\theta}), \qquad \text{and}\qquad h_r(e^{i\theta}) = h(r e^{i\theta})\]
for $0 \leq r \leq \varrho$. The triangle inequality and the Cauchy--Schwarz inequality yield that
\[\|f_r - f_\varrho \|_1 \leq \|g_r h_r - g_\varrho h_r \|_1 + \|g_\varrho h_r - g_\varrho h_\varrho \|_1 \leq \|g_r-g_\varrho\|_2\|h_r\|_2 + \|g_\varrho\|_2 \|h_r-h_\varrho\|_2.\]
Since $g$ and $h$ are analytic for $|z| < \varrho+\varepsilon$, their power series at the origin converge absolutely for $|z| \leq \varrho$. We deduce from this, orthogonality, and the trivial estimate $(r^k-\varrho^k)^2 \leq \varrho^{2k}-r^{2k}$ that
\[\|g_r-g_\varrho\|_2 \leq \sqrt{\|g_\varrho\|_2^2 - \|g_r\|_2^2} \qquad \text{and} \qquad \|h_r-h_\varrho\|_2 \leq \sqrt{\|h_\varrho\|_2^2 - \|h_r\|_2^2}.\]
Putting together what we have done so far, we find that
\[\|f_r - f_\varrho \|_1 \leq \sqrt{\|g_\varrho\|_2^2\|h_r\|_2^2 - \|g_r\|_2^2 \|h_r\|_2^2} + \sqrt{\|h_\varrho\|_2^2\|g_\varrho\|_2^2 - \|g_\varrho\|_2^2 \|h_r\|_2^2}.\]
Since plainly $\|h_r\|_2^2 \leq \|h_\varrho\|_2^2$ and $\|g_\varrho\|_2^2 \geq \|g_r\|_2^2$ by orthogonality, we get that
\[\|f_r - f_\varrho\|_1 \leq 2 \sqrt{\|g_\varrho\|_2^2 \|h_\varrho\|_2^2 - \|g_r\|_2^2 \|h_r\|_2^2.}\]
We use that $|B(z)|=1$ for $|z|=\varrho$ to infer that $\|g_\varrho\|_2^2 = \|f_\varrho\|_1$ and $\|h_\varrho\|_2^2 = \|f_\varrho\|_1$, and the Cauchy--Schwarz inequality to infer that $\|g_r\|_2^2 \|h_r\|_2^2 \geq \|f_r\|_1^2$. \qed

\section{Hilbert's criterion} \label{sec:hilbertcrit} 
We find it necessary to begin with some expository material in order to properly set the stage for the proofs of Theorem~\ref{thm:hilbertcrit} and Theorem~\ref{thm:fmrieszinfty}.

If $K$ is a finite subset of $\mathbb{Z}^{(\infty)}$, then we say that the function 
\begin{equation}\label{eq:trigpoly} 
	T(\chi) = \sum_{\kappa \in K} a_\kappa \chi^\kappa 
\end{equation}
is a \emph{trigonometric polynomial} on $\mathbb{T}^\infty$. It follows from the definition of $\mathbb{Z}^{(\infty)}$ that there is for each trigonometric polynomial $T$ a positive integer $d$ such that $T$ only depends on a subset of the variables $\chi_1,\chi_2,\ldots,\chi_d$.

We let $L^p(\mathbb{T}^d)$ stand for the closed subspace of $L^p(\mathbb{T}^\infty)$ obtained as the closure of the set of such trigonometric polynomials. If $f$ is in $L^p(\mathbb{T}^d)$, then the Fourier coefficients of $f$ are plainly supported on sequences in $\mathbb{Z}^{(\infty)}$ of the form 
\begin{equation}\label{eq:Zd} 
	(\kappa_1,\kappa_2,\ldots,\kappa_d,0,0,\ldots). 
\end{equation}
For $d=1,2,3,\ldots$, die Abschnitte $\mathfrak{A}_d f$ are formally defined as replacing the Fourier coefficient $\widehat{f}(\kappa)$ by $0$ whenever $\kappa$ is not of the form \eqref{eq:Zd}. The following result can be obtained from density and the mean value property of trigonometric polynomials. The proof is not difficult and we omit it.
\begin{lemma}\label{lem:abschnitt} 
	Let $1 \leq p < \infty$. For $d=1,2,3,\ldots$, die Abschnitte $\mathfrak{A}_d$ extend to bounded linear operators from $L^p(\mathbb{T}^\infty)$ to $L^p(\mathbb{T}^d)$ satisfying
	\[\|\mathfrak{A}_1 f\|_p \leq \|\mathfrak{A}_2 f\|_p \leq \|\mathfrak{A}_3 f\|_p \leq \cdots \leq \|f\|_p\]
	for every $f$ in $L^p(\mathbb{T}^\infty)$. 
\end{lemma}

We are now in a position to establish Hilbert's criterion for $L^p(\mathbb{T}^\infty)$, which in particular covers the assertion (i) of Theorem~\ref{thm:hilbertcrit}.
\begin{theorem}\label{thm:hilbertcritLp} 
	Suppose that $1 \leq p < \infty$. If $f$ is in $L^p(\mathbb{T}^\infty)$, then
	\[\lim_{d\to \infty} \|f-\mathfrak{A}_d f\|_p = 0.\]
\end{theorem}
\begin{proof}
	Fix $\varepsilon>0$. By density, we can find a trigonometric polynomial $T$ such that $\|f-T\|_p \leq \varepsilon/2$. Since $T$ is a trigonometric polynomial, there is a positive integer $d_0$ such that $T$ is in $L^p(\mathbb{T}^{d_0})$. It now follows from the triangle inequality and Lemma~\ref{lem:abschnitt} that if $d \geq d_0$, then
	\[\|f-\mathfrak{A}_d f\|_p \leq \|f-T\|_p + \|T-\mathfrak{A}_d f\|_p = \|f-T\|_p + \|\mathfrak{A}_d(T-f)\|_p \leq \varepsilon. \qedhere\]
\end{proof}

We will use a weaker and less attractive version of Lemma~\ref{lem:main} in the proofs of Theorem~\ref{thm:hilbertcrit}~(ii) and Theorem~\ref{thm:fmrieszinfty}. We retain the notation $f_r(e^{i\theta})=f(r e^{i\theta})$ for analytic functions $f$ in $\mathbb{D}$ and $0 \leq r <1$, but write $\|\cdot\|_{L^1(\mathbb{T})}$ to distinguish the norm of $L^1(\mathbb{T})$ from the norm of $L^1(\mathbb{T}^\infty)$.
\begin{lemma}\label{lem:mainadjust} 
	If $f$ is analytic in $\mathbb{D}$ and $0 \leq r \leq \varrho < 1$, then
	\[\|f_r-f_\varrho\|_{L^1(\mathbb{T})} \leq 2\sqrt{2}\sqrt{\|f_\varrho\|_{L^1(\mathbb{T})}}\sqrt{\|f_\varrho\|_{L^1(\mathbb{T})}-\|f_r\|_{L^1(\mathbb{T})}}.\]
\end{lemma}
\begin{proof}
	Use Lemma~\ref{lem:main} and the fact that $b^2-a^2 \leq 2b(b-a)$ for $0\leq a \leq b$. 
\end{proof}

A \emph{polynomial} $P$ on $\mathbb{T}^\infty$ is a trigonometric polynomial \eqref{eq:trigpoly} where the index set $K$ is a subset of $\mathbb{N}_0^{(\infty)}$. Polynomials on $\mathbb{T}^\infty$ are nothing more than classical polynomials in, say, $d$ variables restricted to $(\chi_1,\chi_2,\ldots,\chi_d)$. This means we can extend polynomials on $\mathbb{T}^\infty$ to $\mathbb{C}^\infty$ in the obvious way. In particular, if $d_1\leq d$, then
\[\mathfrak{A}_{d_1} P(\chi) = P(\chi_1,\chi_2,\ldots,\chi_{d_1},0,0,\ldots,0).\]

For the proof of Theorem~\ref{thm:hilbertcrit}~(ii), we will use the following consequence of Lemma~\ref{lem:mainadjust}. The basic idea to embed a \emph{slice} of the disc in a polydisc is from Rudin~\cite{Rudin1969}*{p.~44}.
\begin{lemma}\label{lem:H1abschnitt} 
	If $f$ is in $H^1(\mathbb{T}^\infty)$ and if $d_1 \leq d_2$ are positive integers, then
	\[\|\mathfrak{A}_{d_1}f-\mathfrak{A}_{d_2}f \|_1 \leq 2\sqrt{2} \sqrt{\|\mathfrak{A}_{d_2}f\|_1} \sqrt{\|\mathfrak{A}_{d_2}f\|_1-\|\mathfrak{A}_{d_1}f\|_1}.\]
\end{lemma}
\begin{proof}
	By density and Lemma~\ref{lem:abschnitt}, it is sufficient to establish the stated estimate for polynomials $P$ in $L^p(\mathbb{T}^{d_2})$. In this case, we define
	\[F(\chi,z) = P(\chi_1, \chi_2, \ldots, \chi_{d_1}, \chi_{d_1+1}z,\chi_{d_1+2}z,\ldots,\chi_{d_2} z)\]
	for $\chi$ on $\mathbb{T}^\infty$ and $z$ in $\mathbb{C}$. If $\chi$ is fixed, then $f(z) = F(\chi,z)$ is a polynomial and it is permissible to use Lemma~\ref{lem:mainadjust} with $r=0$ and $\varrho=1$. We next integrate over $\chi$ on $\mathbb{T}^\infty$, then finally use the Cauchy--Schwarz inequality to infer that 
	\begin{multline*}
		\int_{\mathbb{T}^\infty} \|F(\chi,\cdot)-F(\chi,0)\|_{L^1(\mathbb{T})} \,dm_\infty(\chi) \\
		\leq 2\sqrt{2} \sqrt{\int_{\mathbb{T}^\infty} \|F(\chi,\cdot)\|_{L^1(\mathbb{T})} \,dm_\infty(\chi) } \\
		\times \sqrt{\int_{\mathbb{T}^\infty} \big(\|F(\chi,\cdot)\|_{L^1(\mathbb{T})}-|F(\chi,0)|\big)\,dm_\infty(\chi)}. 
	\end{multline*}
	The stated estimate follows from this after using that $F(\chi,0) = \mathfrak{A}_{d_1} P(\chi)$ twice, then using Fubini's theorem with the rotational invariance of $m_\infty$ thrice. 
\end{proof}

Lemma~\ref{lem:H1abschnitt} is the key ingredient in our proof of Theorem~\ref{thm:hilbertcrit}~(ii). The idea to establish Hilbert's criterion via a result such as Lemma~\ref{lem:H1abschnitt} is from \cite{BBSS2019}*{Section~2.2}.
\begin{proof}
	[Proof of Theorem~\ref{thm:hilbertcrit}~\normalfont{(ii)}] If $(f_d)_{d\geq1}$ is a bounded sequence in $H^1(\mathbb{T}^\infty)$ that enjoys the chain property, then it follows from Lemma~\ref{lem:H1abschnitt} that $(f_d)_{d \geq1}$ is a Cauchy sequence in $H^1(\mathbb{T}^\infty)$. Hence it must converge to some function $f$ in $H^1(\mathbb{T}^\infty)$. Fourier coefficients are preserved under convergence in $L^1(\mathbb{T}^\infty)$, so that $\mathfrak{A}_d f = f_d$ for $d=1,2,3,\ldots$. 
\end{proof}

In preparation for the proof of Theorem~\ref{thm:fmrieszinfty}, we recall that a result of Cole and Gamelin \cite{CG1986}*{Theorem~4.1} asserts that the infinite product
\[\prod_{j=1}^\infty \frac{1-|z_j|^2}{|\chi_j - z_j|^2}\]
converges to a bounded function on $\mathbb{T}^\infty$ if and only if $z$ is in $\mathbb{D}^\infty \cap \ell^1$. This means that the Poisson extension
\[\mathfrak{P}\mu(z) = \int_{\mathbb{T}^\infty} \prod_{j=1}^\infty \frac{1-|z_j|^2}{|\chi_j - z_j|^2}\,d\mu(\chi)\]
of a finite complex Borel measure $\mu$ on $\mathbb{T}^\infty$ can in general only be defined in $\mathbb{D}^\infty \cap \ell^1$.

Our final preparation for the proof of Theorem~\ref{thm:fmrieszinfty} is to recall that finite complex Borel measures on $\mathbb{T}^\infty$ are uniquely determined by their Fourier coefficients. As in the classical setting, this is a direct consequence of the Riesz representation theorem and the fact that trigonometric polynomials are dense in $C(\mathbb{T}^\infty)$.
\begin{proof}
	[Proof of Theorem~\ref{thm:fmrieszinfty}] If $\chi$ is on $\mathbb{T}^\infty$, $z$ is in $\mathbb{D}$, and $d$ is a positive integer, then the point $(\chi_1 z, \chi_2 z, \ldots, \chi_d z,0,0,\ldots)$ is plainly in $\mathbb{D}^\infty \cap \ell^1$. We can therefore define
	\[F(\chi,z,d) = \mathfrak{P} \mu(\chi_1 z, \chi_2 z, \ldots, \chi_d z,0,0,\ldots).\]
	Using Fubini's theorem as in the classical setting discussed in the introduction, we get that $\|F(\cdot,\varrho,d)\|_1 \leq \|\mu\|$. If $\chi$ and $d$ are fixed, then this and the assumption on the support of the Fourier coefficients of $\mu$ ensure that $F(\cdot,z,d)$ is in
	\[H^1(\mathbb{T}^d) = H^1(\mathbb{T}^\infty) \cap L^1(\mathbb{T}^d).\]
	This assumption also ensures that if $\chi$ and $d$ are fixed, then $f(z) = F(\chi,z,d)$ is analytic in $\mathbb{D}$. Arguing as in the proof of Lemma~\ref{lem:H1abschnitt}, we infer from Lemma~\ref{lem:mainadjust} that
	\[\|F(\cdot,r,d)-F(\cdot,\varrho,d)\|_1 \leq 2\sqrt{2}\sqrt{\|F(\cdot,\varrho,d)\|_1} \sqrt{\|F(\cdot,\varrho,d)\|_1 - \|F(\cdot,r,d)\|_1}\]
	for $0 \leq r \leq \varrho < 1$. We infer from this that there is a function $f_d$ in $H^1(\mathbb{T}^d)$ with $\|f_d\|_1 \leq \|\mu\|$ such that
	\[\lim_{r\to 1^-} \|f_d -F(\cdot,r,d)\|_1 = 0.\]
	It follows that $(f_d)_{d\geq1}$ is a bounded sequence in $H^1(\mathbb{T}^\infty)$ that enjoys the chain property, so by Theorem~\ref{thm:hilbertcrit}~(ii) there is a function $f$ in $H^1(\mathbb{T}^\infty)$ such that $f_d = \mathfrak{A}_d f$ for $d=1,2,3,\ldots$ and so $f = \mu$ by Theorem~\ref{thm:hilbertcrit}~(i). 
\end{proof}

It is possible to give a slightly different proof of Theorem~\ref{thm:fmrieszinfty} that does not use Hilbert's criterion. The idea (from \cite{AOS2019}) is to consider the Poisson extensions of $\mu$ to the points $(\chi_1 z, \chi_2 z^2, \chi_3 z^3, \ldots)$, which are in $\mathbb{D}^\infty \cap \ell^1$ for $\chi$ on $\mathbb{T}^\infty$ and $z$ in $\mathbb{D}$, and then use Lemma~\ref{lem:mainadjust} as above.

\bibliography{fmriesz}

\end{document}